\documentclass[11pt]{amsart}
\usepackage{geometry}                
\usepackage{graphicx}
\usepackage{amssymb,amscd}
\usepackage{epstopdf,breqn}
\usepackage{tikz}
\usepackage{hyperref}
\usetikzlibrary{arrows, shapes}
\newtheorem{thm}{Theorem}[section]
\newtheorem{lem}{Lemma}[section]
\newtheorem{cor}{Corollary}[section]
\newtheorem{pro}{Proposition}[section]
\theoremstyle{definition}
\newtheorem{definition}{Definition}[section]

\def\rnum#1{\expandafter{\romannumeral #1}}
\def\Rnum#1{\uppercase\expandafter{\romannumeral #1}} 

\newcommand{\supp}{\mathop{\mathrm{supp}}}
\newcommand{\Sw}{\mathcal{S}}

\newcommand{\nrmp}[1]{||#1 ||_{L^p}}

\newcommand{\agl}[1]{\langle#1\rangle}
\newcommand{\Z}{\mathbb{Z}}

\newcommand{\rn}{\mathbb{R}^n}
\newcommand{\rne}{\mathbb{R}^{n-1}}
\newcommand{\T}{\mathbb{T}}
\newcommand{\dis}{\displaystyle}

\newcommand{\vep}{\varepsilon}
\newcommand{\lan}{\langle}
\newcommand{\ran}{\rangle}
\newcommand{\sq}{\square}

\begin{document}
\title{Trace operators on Wiener amalgam spaces}
\author{Jayson Cunanan}
\address{Department of Mathematical Sciences, 
Faculty of Science, 
Shinshu University, Asahi 3-1-1, 
Matsumoto, Nagano, 390-8621, 
Japan}
  \email{jcunanan@shinshu-u.ac.jp}
 
\author{Yohei Tsutsui}
\address{Department of Mathematical Sciences, 
Faculty of Science, 
Shinshu University, Asahi 3-1-1, 
Matsumoto, Nagano, 390-8621, 
Japan}
  \email{tsutsui@shinshu-u.ac.jp}
  
  
\maketitle

\begin{abstract}
The paper deals with trace operators of Wiener amalgam spaces using frequency-uniform decomposition operators and maximal inequalities, obtaining sharp results. Additionally, we provide the embeddings between standard and anisotropic Wiener amalgam spaces.
\end{abstract}
\section{Introduction}
The aim of this paper is to study the trace problem: What can be said about the trace operator $\mathbb{T},$
$$\mathbb{T}:f(x)\rightarrow f(\bar{x},0),\quad \bar{x}=(x_1,x_2,...,x_{n-1}),$$
as a mapping from $W^{p,q}_s(\mathbb{R}^n)$ to $W^{p,q}_s(\mathbb{R}^{n-1}).$ We note that for a tempered distribution $f$ defined on $\mathbb{R}^n, f(x,0)$ has no straightforward meaning and the question is how to define the trace for a class of tempered distributions. One can resort to the Schwartz function $\phi$, which has a pointwise trace $\phi(\bar{x},0)$. It can be extended to (quasi-)Banach function spaces which contain the Schwartz space $\mathcal{S}$ as a dense subspace.

Our setting is on Wiener amalgam spaces. These spaces, together with modulation spaces, were introduced by Feichtinger \cite{f1,f2,abl} in the 80's and are now widely used function spaces for various problems in PDE and harmonic analysis \cite{ben,rem,cun2,GC99,wang}. They resemble Triebel-Lizorkin spaces in the sense that we are taking $L^p(\ell^q)$ norms, but differ with the decomposition operator being used. Instead of the dyadic decomposition operators $\Delta_k\sim\mathcal{F}^{-1}\chi_{\{\xi:|\xi|\sim2^k\}}\mathcal{F}$ used for Triebel-Lizorkin spaces, Wiener amalgam spaces use frequency uniform decomposition operators $\square_k\sim\mathcal{F}^{-1}\chi_{Q_k}\mathcal{F}$, where $Q_k$ denotes a unit cube with center $k$ and $\cup_{k\in\mathbb{Z}^n}Q_k=\mathbb{R}^n.$

The concept of trace operator plays an important role in studying the existence and uniqueness of solutions to boundary value problems, that is, to partial differential equations with prescribed boundary conditions \cite{law,tri}. The trace operator makes it possible to extend the notion of restriction of a function to the boundary of its domain to "generalized" functions in various function spaces with regulariy. Now, we give a formal definition for the trace operators.

\begin{definition}
Let $X$ and $Y$ be quasi-Banach function spaces defined on $\rn$ and $\rne$, respectively. Assume that the Schwartz class $\mathcal{S}$ is dense in $X.$ Denote
$$\mathbb{T}:f(x)\rightarrow f(\bar{x},0),\quad f\in\mathcal{S}.$$
Assuming that there exist a constant $C>0$ such that
$$||\T f||_Y\leq C||f||_X,\quad \forall f\in\mathcal{S},$$
one can extend $\T:X\rightarrow Y$ by the density of $\Sw$ in $X$ and we write $f(\bar{x},0)=\T f,$ which is said to be the trace of $f\in X$. Moreover, if there exist a continuous linear operator $\T^{-1}:Y\rightarrow X$ such that $\T\T^{-1}$ is the identity operator on Y, then $\T$ is said to be a trace-retraction from $X$ onto $Y.$
\end{definition}
For ($\alpha$-) modulation spaces, Besov spaces and Tribel-Lizorkin spaces, trace theorems have been extensively studied \cite{fwh,trimod,tri}. Feichtinger, Huang and Wang \cite{fwh} considered the trace theorems on anisotropic modulation spaces $M^{p,q,r}_s$ with $0< p,q,r< \infty,s\in\mathbb{R}$ and they obtained $\T M_s^{p,q,p\wedge q\wedge1}(\rn)=M^{p,q}_s(\rne)$. In \cite{fra,tri2}, we find that for $0<p,q\leq\infty,$ and $s-1/p>(n-1)(1/p-1)$, we have $\T B_s^{p,q}(\rn)=B^{p,q}_{s-1/p}(\rne)$ and $\T F_s^{p,q}(\rn)=F^{p,p}_{s-1/p}(\rne)$ (the case $F^{\infty,q}$ is ommited). The use of atoms as a framework in studying trace problems can be found in \cite{tri2} and the references within.

Our main results are the following.
\begin{thm}
Let $n\geq2, 0<p,q< \infty,s\in\mathbb{R}.$ Then
$$\mathbb{T}:f(x)\rightarrow f(\bar{x},0),\quad \bar{x}=(x_1,x_2,...,x_{n-1})$$
is a trace-retraction from $W^{p,q,1\wedge q}_s(\mathbb{R}^n)$ to $W^{p,q}_s(\mathbb{R}^{n-1}).$
\end{thm}
In view of  the embedding in Theorem 2.1 (II-ii),  we immediately have the following corollary.   
\begin{cor}
Let $n\geq2, 0<p,q< \infty,s\geq0.$ Then for any $\epsilon>0$
$$\mathbb{T}:W^{p,q}_{s+\frac{1}{1\wedge q}-\frac{1}{q}+\epsilon}(\mathbb{R}^n)\rightarrow W^{p,q}_{s}(\mathbb{R}^{n-1}).$$
\end{cor}

We remark that our result shows independence of $p.$ This is due the pointwise estimates we were able to prove in Section 3. An interesting observation is that, the trace theorem of Triebel-Lizorkin spaces stated above, shows independence in $q.$ This difference might be due to the decomposition operators used in the norm of each function spaces.

The paper is organised as follows: In Section 2, the embeddings between standard and anisotropic Wiener amalgam spaces are given. We also define notations, function spaces and some Lemmas to be used throughout this paper. In Section 3, we prove our main result, Theorem 1.1  and the sharpness of Corollary 1.1.
\section{Preliminaries}
\textbf{Notations.} The Schwartz class of test functions on $\mathbb{R}^n$ shall be denoted by $\mathcal{S}:=\mathcal{S}(\mathbb{R}^n)$ and its dual, the space of tempered distributions, by $\mathcal{S}':=\mathcal{S}'(\mathbb{R}^n)$. The $L^p(\mathbb{R}^n)$ norm is given by $\nrmp{f}=(\int_{\mathbb{R}^n} |f(x)|^p\ dx)^{1/p} $ whenever $1\leq p<\infty,$ and $||f||_{L^{\infty}}=\text{ess.sup}_{x\in\mathbb{R}^n}|f(x)|$. The Fourier transform of a function $f\in \mathcal{S}(\mathbb{R}^n)$ is given by $$\mathcal{F}f(\xi)=\hat{f}(\xi)=\int_{\mathbb{R}^n}e^{-i2\pi x\cdot\xi}f(x)\ dx$$ which is an isomorphism of the Schwartz space $\mathcal{S}(\mathbb{R}^n)$ onto itself that extends to the tempered distributions $\mathcal{S'}(\mathbb{R}^n)$ by duality. The inverse Fourier transform is given by $\mathcal{F}^{-1}f(x)=\check{f}(x)=\int_{\mathbb{R}^n}e^{i2\pi \xi\cdot x}f(\xi)\ d\xi$. Given $1\leq p\leq\infty,$ we denote by $p'$ the conjugate exponent of $p$ (i.e. $1/p+1/p'=1$). We use the notation $u\lesssim v$ to denote $u\leq cv$ for a positive constant $c$ independent of $u$ and $v$. We write $a\wedge b:=\min(a,b)$ and $a\vee b:=\max(a,b)$. We now define the function spaces in this paper.

Let $\eta:\mathbb{R}\rightarrow[0,1]$ be a smooth bump function satisfying 
\begin{equation*}\label{eta}
\eta(\xi):=\begin{cases}
      1, & |\xi|\leq1 \\
       smooth, & 1<|\xi|\leq 2\\
      0, & |\xi|\geq2.
    \end{cases}
\end{equation*}
We write for $k=(k_1,...,k_n)$ and $\xi=(\xi_1,...,\xi_n),$
    $$\phi_{k_i}=\eta(2(\xi_i-k_i)).$$ 
 Put 
 \begin{equation}\label{vph}
\varphi_k(\xi)=\dfrac{\phi_{k_1}(\xi_1)...\phi_{k_n}(\xi_n)}{\sum_{k\in\mathbb{Z}^n}\phi_{k_1}(\xi_1)...\phi_{k_n}(\xi_n)},\quad k\in\mathbb{Z}^n.
\end{equation}

\begin{definition}[\textit{Wiener amalgam spaces}]
For $ 0< p,q\leq \infty,$ and $ s\in\mathbb{R}$, the Wiener amalgam space $ W^{p,q}_s $ consists of all tempered distributions $f\in \mathcal{S}'$ for which the following is finite:
\begin{equation}\label{qua}
||f||_{W^{p,q}_s}=||\ ||\{\agl{k}^{s}\square_kf\}||_{\ell^q}||_{L^p},
\end{equation}
with
$\square_kf={\mathcal F}^{-1}( \varphi_k\widehat{f})$.
\end{definition}
We note that (\ref{qua}) is a quasi-norm if $ 0< p,q\leq \infty,$ and norm if $1\leq p,q\leq\infty$. Moreover, (\ref{qua}) is independent of the choice of $\varphi=\{ \varphi_k\}_{k\in\Z^n}$. We refer the reader to \cite{f1,f2,f3} for equivalent definitions (continuous versions).

We write $\bar{x}=(x_1,x_2,...,x_{n-1})$  and define the anisotropic Wiener amalgam spaces $W^{p,q,r}$ by the following norm,

$$||f||_{W^{p,q,r}_s(\mathbb{R}^n)}=||(\sum_{k_n\in\mathbb{Z}}(\sum_{\bar{k}\in\mathbb{Z}^{n-1}}\agl{\bar{k}}^{sq}|\square_kf|^q)^{r/q})^{1/r}||_{L^p(\mathbb{R}^n)}.$$

Similarly, for $\bar{\bar{x}}=(x_1,x_2,...,x_{n-2}),$ we define 

$$||f||_{W^{p,q,r,r}_s(\mathbb{R}^n)}=||(\sum_{(k_{n-1}k_n)\in\mathbb{Z}^2}(\sum_{\bar{\bar{k}}\in\mathbb{Z}^{n-2}}\agl{\bar{\bar{k}}}^{sq}|\square_kf|^q)^{r/q})^{1/r}||_{L^p(\mathbb{R}^n)}.$$

Comparing amalgam spaces $W_s^{p,q}$ with anisotropic amalgam spaces $W_s^{p,q,r}$ we see that $W_s^{p,q}$ is, but $W_s^{p,q,r}$ is not rotational invariant. Using the almost orthogonality of $\varphi$ we see that the $W_s^{p,q,r}$ is independent of $\varphi$. Moreover, recalling that $||f||_{W_s^{p,q,r}}$ is the function
sequence $\{\square_k f \}_{k\in\mathbb{Z}^n}$ equipped with the $L^p\ell^r_{k_n}\ell^q_{\bar{k}}$ norm, it is easy to see that $W_s^{p,q,r}$ is a quasi-Banach space for any $s\in\mathbb{R},
p,q,r \in (0,\infty]$ and a Banach space for any $s\in\mathbb{R},
1\leq p,q,r\leq \infty.$ Moreover, the Schwartz space is dense in $W_s^{p,q,r}$ if $p, q, r < \infty.$ The proofs are similar to those of amalgam spaces in \cite{f1,f2,f3}. 


We collect properties of Wiener amalgam spaces in the following lemma. 
\begin{lem}
Let $p,q,p_i.q_i\in[1,\infty]$ for $i=1,2$ and $s_j\in\mathbb{R}$ for $\ j=1,2.$ Then

\begin{enumerate}
\item $\mathcal{S}(\mathbb{R}^n)\hookrightarrow W^{p,q}(\mathbb{R}^n)\hookrightarrow \mathcal{S'}(\mathbb{R}^n);$
\item $\mathcal{S}$ is dense in $W^{p,q}$ if $p$ and $q<\infty;$
\item If $q_1\leq q_2$ and $p_1\leq p_2$, then $W^{p_1,q_1}\hookrightarrow W^{p_2,q_2};$
\item If $s_1\geq s_2$, then $W^{p,q}_{s_1}\hookrightarrow W^{p,q}_{s_2};$
\item (Complex interpolation) For $0<\theta<1.$ Let $\dfrac{1}{p}=\dfrac{\theta}{p_1}+\dfrac{1-\theta}{p_2}$, $\dfrac{1}{q}=\dfrac{\theta}{q_1}+\dfrac{1-\theta}{q_2}$ and $s=\theta s_1+ (1-\theta)s_2.$ Then
$$[{W}^{p_1,q_1}_{s_1}, {W}^{p_2,q_2}_{s_2}]_{[\theta]}={W}^{p,q}_s.$$
\end{enumerate}
\end{lem}
The proofs of these statements can be found in \cite{f1,abl,f3,tof2}. 

\begin{thm}[Embedding: $W^{p,q}_s \hookrightarrow W^{p,q,r}_{s^\prime}$]
Let $p,q,r \in (0,\infty]$ and $s \ge 0$.\\
(\Rnum{1}): The case $r=q$.

(\Rnum{1} - \rnum{1}): The case $r=q=\infty$.
\[
W^{p,\infty}_s \hookrightarrow W^{p,\infty,\infty}_s.
\]

(\Rnum{1} - \rnum{2}): The case $r=q < \infty$.
\[
W^{p,q}_s \hookrightarrow W^{p,q,q}_s.
\]

\vspace{5mm}

\noindent
(\Rnum{2}): The case $r<q$.

(\Rnum{2} - \rnum{1}): The case $q=\infty$.
If $s > 1/r$, then
\[
W^{p,\infty}_s \hookrightarrow W^{p,\infty,r}_{s^\prime},
\]
for any $s^\prime \in (-\infty, s-1/r)$.

(\Rnum{2} - \rnum{2}): The case $q < \infty$.
If $s > (1/r - 1/q)$, then
\[
W^{p,q}_s \hookrightarrow W^{p,q,r}_{s^\prime},
\]
for any $s^\prime \in (-\infty, s-(1/r - 1/q))$.

\vspace{5mm}

\noindent
(\Rnum{3}): The case $q<r$.

(\Rnum{3} - \rnum{1}): The case $r=\infty$.
\[
W^{p,q}_s \hookrightarrow W^{p,q,\infty}_s.
\]

(\Rnum{3} - \rnum{2}): The case $r<\infty$.
\[
W^{p,q}_s \hookrightarrow W^{p,q,r}_s.
\]
\end{thm}

\begin{proof}
For part (I), it suffice to show the following estimates.

(\Rnum{1}-\rnum{1}): 
\[
\dis \sup_{k_n} \sup_{\bar{k}} \lan \bar{k} \ran^s |\sq_kf| \le \sup_{k} \lan k \ran^s |\sq_kf|.
\]

\medskip

(\Rnum{1} - \rnum{2}):
\[
\dis \left(\sum_{k_n} \sum_{\bar{k}} \lan \bar{k} \ran^{sq} |\sq_k f|^q  \right)^{1/q} \le \left(\sum_k \lan k \ran^{sq} |\sq_k f|^q \right)^{1/q}.
\]

\medskip

(\Rnum{2} - \rnum{1}): Let $s' := s - 1/r - \vep,\ (\vep >0)$.
We may assume that $s^\prime \ge 0$.
\[
\dis \left(\sum_{k_n} \sup_{\bar{k}} \lan \bar{k} \ran^{s^\prime r} |\sq_k f|^{r} \right)^{1/r} \le \left(\sup_k \lan k \ran^{s} |\sq_k f| \right) \times \left(\sum_{k_n} \sup_{\bar{k}} \lan k \ran^{-sr} \lan \bar{k} \ran^{s^\prime r} \right)^{1/r} 
\]
The last term is equivalent to
\[
\dis \left(\sum_{m \in \Z} \left(\sup_{t \ge 1} \left(\dfrac{1}{t + |m|} \right)^s t^{s^\prime} \right)^r \right)^{1/r} \le \left(\sum_{m \in \Z} \left(\dfrac{1}{1 + |m|} \right)^{(s - s^\prime) r} \sup_{t \ge 1} t^{(s^\prime - s^\prime) r} \right)^{1/r} < \infty,
\]
here $(s - s^\prime)r = 1 + \vep r >1$ and $s^\prime \ge 0$ have been used.

\medskip

(\Rnum{2} - \rnum{2}): Let $s^\prime := s - (1/r - 1/q) - \vep,\ (\vep >0)$.
It suffice to show the embedding in the case $s^\prime \ge 0$.
Remark that $q/r \in (1,\infty)$ and $(q/r)^\prime = 1/(r(1/r - 1/q))$.
Let $\alpha := 1 - r/q + \vep r$.

\begin{eqnarray*}
&&\dis \left[\sum_{k_n} \left\{\sum_{\bar{k}} \lan \bar{k} \ran^{s^\prime q} |\sq_k f|^q \right\}^{r/q} \right]^{1/r}\\ && = \left[\sum_{k_n} \left\{\sum_{\bar{k}} \lan \bar{k} \ran^{s^\prime q} \lan k_n \ran^{\alpha q/r} |\sq_k f|^q \right\}^{r/q} \lan k_n \ran^{- \alpha} \right]^{1/r} \\
& & \le \left[\left\{\sum_{k_n} \sum_{\bar{k}} \lan \bar{k} \ran^{s^\prime q} \lan k_n \ran^{\alpha q/r} |\sq_k f|^q \right\}^{r/q}\times \left(\sum_{k_n} \lan k_n \ran^{- \alpha (q/r)^\prime} \right)^{1/(q/r)^\prime} \right]^{1/r} \\
 && \lesssim \left(\sum_k \lan \bar{k} \ran^{s^\prime q} \lan k_n \ran^{\alpha q/r} |\sq_k f|^q \right)^{1/q} \\
& & = \left\{\sum_k \lan k \ran^{sq} |\sq_k f|^q \left(\lan \bar{k} \ran^{s^\prime} \lan k_n \ran^{\alpha /r} \lan k \ran^{-s} \right)^q \right\}^{1/q} \\
& & \le \left[\sup_k \lan \bar{k} \ran^{s^\prime} \lan k_n \ran^{\alpha /r} \lan k \ran^{-s} \right] \times \left\{\sum_k \lan k \ran^{sq} |\sq_k f|^q \right\}^{1/q}.
\end{eqnarray*}
Here, we have used $\alpha (q/r)^\prime = 1+ \dfrac{\vep}{1/r - 1/q} >1$.
Because $\alpha/r = 1/r - 1/q + \vep = s - s^\prime,\ s - s^\prime \ge 0$ and $s^\prime \ge 0$,
\begin{align*}
\lan \bar{k} \ran^{s^\prime} \lan k_n \ran^{\alpha /r} \lan k \ran^{-s} = \left(\dfrac{\lan k_n \ran}{\lan k \ran} \right)^{s - s^\prime} \left(\dfrac{\lan \bar{k} \ran}{\lan k \ran} \right)^{s^\prime} \lesssim 1.
\end{align*}

\medskip

(\Rnum{3} - \rnum{1}):
\[
\dis \sup_{k_n} \left(\sum_{\bar{k}} \lan \bar{k} \ran^{sq} |\sq_k f|^q \right)^{1/q} \le \left(\sum_k \lan k \ran^{sq} |\sq_k f|^q \right)^{1/q}.
\]
Here, we have used $s \ge 0$.

\medskip

(\Rnum{3} - \rnum{2}):
Using the embedding $\ell^q \hookrightarrow \ell^r$,
\begin{align*}
\dis \left\{\sum_{k_n} \left(\sum_{\bar{k}} \lan \bar{k} \ran^{sq} |\sq_k f|^q \right)^{r/q} \right\}^{1/r} & \le \left(\sum_k \lan\bar{k} \ran^{sq} |\sq_k f|^q \right)^{1/q} \\
& \le  \left(\sum_k \lan k \ran^{sq} |\sq_k f|^q \right)^{1/q}.
\end{align*}
In the last inequality, we need $s \ge 0$.

\end{proof}

\begin{lem}[Triebel, \cite{tri}]\label{tri}
Let $0<p<\infty$ and $0<q\leq\infty$. Let $\Omega=\{\Omega_k\}_{k\in\mathbb{Z}^n}$ be a sequence of compact subsets of $\mathbb{R}^n.$ Let $d_k$ be the diameter of $\Omega_k.$ If $0<r<\min(p,q)$, then there exist a constant $c$ such that
$$|| \sup_{z\in\mathbb{R}^n} \dfrac{|f_k(\cdot-z)|}{1+|d_kz|^{n/r}}  ||_{L^p(\ell^q)}\leq c||  f_k   ||_{L^p(\ell^q)}$$
holds for all $f\in L^p_{\Omega}(\ell^q)$, where $f=\{f_k\}$, $||  f_k   ||_{L^p(\ell^q)}=||\ ||f_k(\cdot)||_{\ell^q}||_{L^p}$ and
$$L^p_{\Omega}(\ell^q)=\{f \ | \ f=\{f_k\}_{k\in\mathbb{Z}^n}\subset\mathcal{S}', \supp \mathcal{F}f_k\subset\Omega_k, \and\ ||f_k||_{L^p(\ell^q)}<\infty\}.$$
\end{lem}

\begin{definition}[\textit{Maximal Functions}]
Let $b>0$ and $f\in\mathcal{S}$. Then
\begin{equation}\label{maxl}
\square^{\ast}_kf(x):= \sup_{y\in\mathbb{Z}^n} \dfrac{|\square_kf(x-y)|}{1+|y|^{b}}\quad x\in\mathbb{R}^n,k\in\mathbb{Z}^n
\end{equation}
\end{definition}

\begin{pro}\label{pro}
Let $0<p<\infty$ and $0<q\leq\infty$, $b>\dfrac{n}{\min (p,q)}$. Then
\begin{equation}
||\left(\sum_{k\in\mathbb{Z}^{n}}\agl{k}^{sq}|\square^{\ast}_kf|^q\right)^{1/q}||_{L^p(\mathbb{R}^n)}
\end{equation}
\and \begin{equation}
||\left(\sum_{k_n\in\mathbb{Z}}\left(\sum_{\bar{k}\in\mathbb{Z}^{n-1}}\agl{\bar{k}}^{sq}|\square^{\ast}_kf|^q\right)^{r/q}\right)^{1/r}||_{L^p(\mathbb{R}^n)}
\end{equation}
are equivalent norms in $W^{p,q}_{s}(\mathbb{R}^n)$ and $W^{p,q,r}_{s}(\mathbb{R}^n)$, respectively.
\end{pro}

The proof is a direct consequence of Lemma \ref{tri}, taking $f_k=\square_kf$. See also \cite[Proposition]{trimod}.

\section{Proof of the main results}

First,we narrate the idea of the proof. We give an equivalent formulation for $\square_{\bar{k}}(\T f)(\bar{x})$, a function in $\rne$, via some $\square_{\bar{k},l}f(\bar{x},0)$ a function in $\rn.$ Then we compute for pointwise estimates between the corresponding $\ell^q$ norms and $\ell^r_{k_n}\ell^q_{\bar{k}}$ norms for cases $0<q<1$ and $1\leq q<\infty,$ separately. Finally, taking $L^p(\rne)$ norms and using our equivalent norms in Proposition  (\ref{pro}), we arrive to our conclusion.

We denote $\mathcal{F}_{\bar{x}}(\mathcal{F}^{-1}_{\bar{\xi}})$ the partial (inverse) Fourier transform on $\bar{x}$ $(\bar{\xi})\in\mathbb{R}^{n-1}$. Write $\{\varphi_{\bar{k}}\}_{\bar{k}\in\Z^{n-1}}$ as versions of ($\ref{vph}$) in $\rne$. By the support property of $\varphi_{\bar{k}}$, we observe
\begin{align}\label{1}
\square_{\bar{k}}(\T f)(\bar{x})&=(\mathcal{F}^{-1}_{\bar{\xi}}\varphi_{\bar{k}}\mathcal{F}_{\bar{x}})(\T f)(\bar{x})\nonumber\\
&=\sum_{l\in\mathbb{Z}^n}\{\mathcal{F}^{-1}_{\bar{\xi}}\varphi_{\bar{k}}\mathcal{F}_{\bar{x}}[(\mathcal{F}^{-1}\varphi_l\mathcal{F}f)(\bar{y},0)]\}(\bar{x})\nonumber\\
&=\sum_{l\in\mathbb{Z}^n}\chi_{(|\bar{k}-\bar{l}|\leq1)}\left(\mathcal{F}^{-1}\psi_{\bar{k},l}\mathcal{F}f\right)(\bar{x},0)\nonumber\\
&=\sum_{l\in\mathbb{Z}^n}\chi_{(|\bar{k}-\bar{l}|\leq1)}\square_{\bar{k},l}f(\bar{x},0),
\end{align}
where $\psi_{\bar{k},l}(\xi)=\varphi_{\bar{k}}(\bar{\xi})\varphi_l(\xi), l=(\bar{l},l_n),$ and $\square_{\bar{k},l}f:=\mathcal{F}^{-1}\psi_{\bar{k},l}\mathcal{F}f$. Note that the left-hand side is a function in $\mathbb{R}^{n-1}$ while the right-hand side is a function in $\mathbb{R}^{n}$.

Recall our maximal function (\ref{maxl}) and take $y_1=y_2=\cdots =y_{n-1}=0, y_n=x_n$ we have for $|x_n|\leq1,$
\begin{equation}\label{max}
|\square_kf(\bar{x},0)|\lesssim \square^{\ast}_kf(x).
\end{equation}

\begin{proof}[Proof of Theorem 1.1]
We start by taking the $\ell^q$-norm of (\ref{1}). We write,
\begin{align}\label{step}
\left(\sum_{\bar{k}\in\mathbb{Z}^{n-1}}\agl{\bar{k}}^{sq}|\square_{\bar{k}}(\T f)(\bar{x})|^q\right)^{1/q}&=\left(\sum_{\bar{k}\in\mathbb{Z}^{n-1}}\agl{\bar{k}}^{sq}\left(\sum_{l\in\mathbb{Z}^n}\chi_{(|\bar{k}-\bar{l}|\leq1)}\square_{\bar{k},l}f(\bar{x},0)\right)^q\right)^{1/q}.
\end{align}

For $0<q<1,$ we estimate (\ref{step}) by

\begin{align}
\left(\sum_{\bar{k}\in\mathbb{Z}^{n-1}}\agl{\bar{k}}^{sq}|\square_{\bar{k}}(\T f)(\bar{x})|^q\right)^{1/q}&\lesssim\left(\sum_{l\in\mathbb{Z}^n}\sum_{\bar{k}\in\mathbb{Z}^{n-1}}\agl{\bar{l}}^{sq}\chi_{(|\bar{k}-\bar{l}|\leq1)}|\square_{\bar{k},l}f(\bar{x},0)|^q\right)^{1/q}\nonumber\\
&=\left(\sum_{l_n\in\mathbb{Z}^{}}\sum_{\bar{l}\in\mathbb{Z}^{n-1}}\agl{\bar{l}}^{sq}\sum_{\bar{k}\in\mathbb{Z}^{n-1}}\chi_{(|\bar{k}-\bar{l}|\leq1)}|\square_{\bar{k},l}f(\bar{x},0)|^q\right)^{1/q}\label{9}
\end{align}
Note that $\sum_{\bar{k}\in\mathbb{Z}^{n-1}}\chi_{(|\bar{k}-\bar{l}|\leq1)}|\square_{\bar{k},l}f(\bar{x},0)|^q=\sum_{j=1}^{n-1}|\square_{\overline{l\pm e_j},l}f(\bar{x},0)|^q$, where $e_j$ is the $j^{th}$ column of the identity matrix. In the sequel, it suffice to consider only the case $j=1.$ Moreover, we write $\widetilde{\square}_lf:=\square_{\overline{l\pm e_1},l}f$ for some $\psi_l$ satisfying (\ref{vph}). Using (\ref{max}) we have,

\begin{align}
\left(\sum_{l_n\in\mathbb{Z}^{}}\sum_{\bar{l}\in\mathbb{Z}^{n-1}}\agl{\bar{l}}^{sq}|\square_{\overline{l\pm e_1},l}f(\bar{x},0)|^q\right)^{1/q}
&\lesssim\left(\sum_{l_n\in\mathbb{Z}^{}}\sum_{\bar{l}\in\mathbb{Z}^{n-1}}\agl{\bar{l}}^{sq}|\widetilde{\square}_l^{\ast }f(\bar{x},x_n)|^q\right)^{1/q}.\label{kk}
\end{align}
Combining (\ref{9}) and (\ref{kk}), then taking the $L^p(\mathbb{R}^{n-1})$-norm and raising to $p$-th power gives,
\begin{align*}
||f(\bar{x},0)||^p_{W^{p,q}_s(\mathbb{R}^{n-1})}&\lesssim||\left(\sum_{l\in\mathbb{Z}^{n}}\agl{\bar{l}}^{sq}\widetilde{\square}_l^{\ast q}f(\bar{x},x_n)\right)^{1/q}||^p_{L^p(\mathbb{R}^{n-1})}
\end{align*}
Integrating over $x_n\in [0,1]$,
\begin{align*}
||f(\bar{x},0)||_{W^{p,q}_s(\mathbb{R}^{n-1})}&\lesssim||\left(\sum_{l\in\mathbb{Z}^{n}}\agl{\bar{l}}^{sq}\widetilde{\square}_l^{\ast q}f(\bar{x},x_n)\right)^{1/q}||_{L^p(\mathbb{R}^{n})}\\
&\lesssim ||f||_{W^{p,q,q}_s(\mathbb{R}^{n})}.
\end{align*}\medskip
Note that the last inequality follows from Proposition 2.1.\medskip

For $1\leq q\leq\infty$, we use Minkowski's inequality to give an upper bound of (\ref{step}) as follows,
\begin{align}
\left(\sum_{\bar{k}\in\mathbb{Z}^{n-1}}\agl{\bar{k}}^{sq}|\square_{\bar{k}}(\T f)(\bar{x})|^q\right)^{1/q}&\lesssim\left(\sum_{\bar{k},\bar{l}\in\mathbb{Z}^{n-1}}\agl{\bar{k}}^{sq}\left(\sum_{l_n\in\mathbb{Z}}\chi_{(|\bar{k}-\bar{l}|\leq1)}\square_{\bar{k},l}f(\bar{x},0)\right)^q\right)^{1/q}\nonumber\\
&\lesssim \sum_{l_n\in\mathbb{Z}} \left(\sum_{\bar{k},\bar{l}\in\mathbb{Z}^{n-1}}\agl{\bar{k}}^{sq}\chi_{(|\bar{k}-\bar{l}|\leq1)}|\square_{\bar{k},l}f(\bar{x},0)|^q\right)^{1/q}\nonumber\\
&\lesssim\sum_{l_n\in\mathbb{Z}} \left(\sum_{\bar{l}\in\mathbb{Z}^{n-1}}\agl{\bar{l}}^{sq}|\widetilde{\square}_l^{\ast}f(\bar{x},x_n)|^q\right)^{1/q}.\label{jj}
\end{align}
Repeating the arguments above on  (\ref{jj}) gives us the estimate,
\begin{align*}
||f(\bar{x},0)||_{W^{p,q}_s(\mathbb{R}^{n-1})}\lesssim ||f||_{W^{p,q,1}_s(\mathbb{R}^{n})}.
\end{align*}
Hence, we arrive to our desired estimates.\medskip

Let $\eta'\in\Sw(\mathbb{R})$ be a function with $\supp\eta'\subset(-1/4,1/4)$ and $(\mathcal{F}^{-1}_{\xi_n})\eta'(0)=1.$ For any $f\in W^{p,q}_s(\mathbb{R}^{n-1})$, we define
$g(x)=(\T^{-1}f)(x):=\left[(\mathcal{F}^{-1}_{\xi_n})\eta'(x_n)\right]f(\bar{x}).$ We easily see that $g(\bar{x},0)=f(\bar{x})$ and $\square_kg=0$ when $|k_n|\geq3$. Moreover, we can decompose $\square_kg=\square_{\bar{k}}f\cdot \square_{k_n} (\mathcal{F}^{-1}_{\xi_n}\eta')$ due to the way $\varphi_k$ is defined in (\ref{vph}). Now we do an estimate,

\begin{align*}
||g||_{W^{p,q,q\wedge1}_s(\mathbb{R}^{n})}&=|| \left(\sum_{k_n\in\mathbb{Z}}\left(\sum_{\bar{k}\in\mathbb{Z}^{n-1}}\agl{\bar{k}}^{sq}|\square_kg|^q\right)^{q\wedge1/q}\right)^{1/ q\wedge1}  ||_{L^p(\mathbb{R}^{n})}\\
&= || \left(\sum_{\bar{k}\in\mathbb{Z}^{n-1}}\agl{\bar{k}}^{sq}|\square_{\bar{k}}f|^q\right)^{1/q}\left(\sum_{|k_n|\leq2}|\square_{k_n} (\mathcal{F}^{-1}_{\xi_n}\eta')|^{1\wedge q}\right)^{1/ q\wedge1}||_{L^p(\mathbb{R}^{n})}\\
&\lesssim ||f||_{W^{p,q}_s(\mathbb{R}^{n-1})}.
\end{align*}

Thus, $\mathbb{T}^{-1}:W^{p,q}_s(\mathbb{R}^{n-1})\rightarrow W^{p,q,q\wedge1}_s(\mathbb{R}^{n})$.

\end{proof}

As the end of this paper, we discuss the optimality of Corollary 1.1. We recall the counterexample given in  \cite{fwh}. For $1<p,q<\infty$, there exist a function which shows
$$\T:M^{p,q}_{1/q'}(\rn)\not\to M^{p,q}_{0}(\rne).$$
 Since $M^{q,q}=W^{q,q},$ we also have $\T:W^{q,q}_{1/q'}(\rn)\not\to W^{q,q}_{0}(\rne)$. Hence, Corollary 1.1 is sharp for $p=q, 1<p,q<\infty$ (refer to FIGURE 1). We now claim that it is also sharp for all $1<p,q<\infty$. Contrary to our claim, suppose $s=1/q'$ implies $\T W^{p,q}_{s}(\rn)=W^{p,q}(\rne)$. Then, by interpolation with the estimate for a point $Q(p_1,q_1)$ with $s=1/q_1'$, one would obtain an improvement for the segment connecting $P(p,q)$ and $Q(p_1,q_1)$ (refer to FIGURE 2), which is not possible. 
 
\begin{figure}
\begin{tikzpicture}[scale=2]
    \draw [<->,thick] (0,2) node (yaxis) [left] {$1/q$}
        |- (2,0) node (xaxis) [below] {$1/p$};
\path[draw] (0,1) node[left] {1} -- (1,1)--(2,1) ;

\node[below]  at (1,0){$1$};

\node[below] at (0,0) {0};

\node at (1,1.5) {$0$};
\node at (1,.65) {$1/q'$};
\end{tikzpicture}
\quad
\quad
\begin{tikzpicture}[scale=2]
    \draw [<->,thick] (0,2) node (yaxis) [left] {$1/q$}
        |- (2,0) node (xaxis) [below] {$1/p$};
\path[draw] (0,1) node[left] {1} -- (1,1)--(2,2) ;
\path[draw] (1,0)--(1,1);
\node[below]  at (1,0){$1$};

\node[below] at (0,0) {0};

\node at (.5,1.5) {$0$};
\node at (.5,.65) {$1/q'$};
\node at (1.5,.65) {$1/p-1/q$};
\end{tikzpicture}\caption{Comparison between the critical regularity index $s$ for $\T W^{p,q}_{s+\epsilon}(\rn)=W^{p,q}(\rne)$ (left) and $\T M^{p,q}_{s+\epsilon}(\rn)=M^{p,q}(\rne)$ (right).}
\label{fig1}
\end{figure}
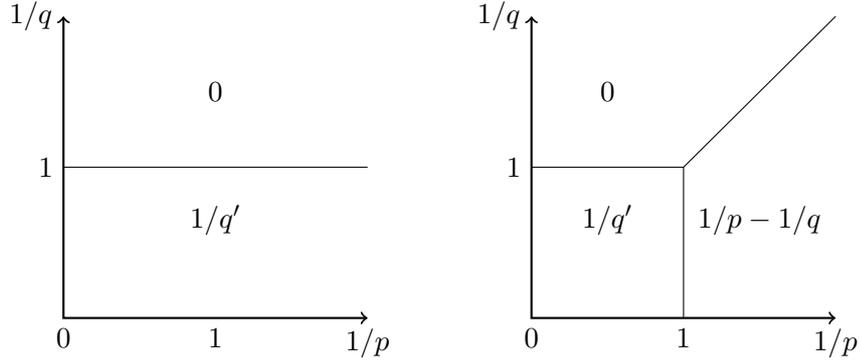

\begin{figure}

\begin{tikzpicture}[scale=2]
    \draw [<->,thick] (0,2) node (yaxis) [left] {$1/q$}
        |- (2,0) node (xaxis) [below] {$1/p$};
\path[draw] (0,1.5) node[left] {1} -- (1.5,1.5)--(2,2) ;
\path[draw] (1.5,0)--(1.5,1.5);
\node[below]  at (1,0){$1$};
\path[draw] (0,0)--(1.5,1.5);

\node[below] at (0,0) {0};

\node at (.5,1.75) {$0$};
\path[draw] (.5,1)node[left] {P}--(1,.55)node[right] {Q};

\end{tikzpicture}\caption{Contradiction argument using interpolation.}
\label{fig1}
\end{figure}
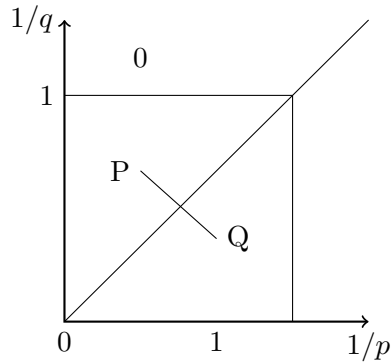

\section*{Acknowledgement}
This work was supported by JSPS, through "Program to Disseminate Tenure Tracking System".
The second author was also partially supported by JSPS, through Grant-in-Aid
for Young Scientists (B) (No. 15K20919).
\bibliographystyle{abbrv}
\bibliography{research}

\end{document}